\def\TSP{Travelling Salesman Problem}
\def\TSPs{Travelling Salesman Problems}
\def\mytitle{On the Solution of the \TSP{} for Nonlinear Salesman Dynamics using Symbolic Optimal Control}
\def\myname{Alexander Weber and Alexander Knoll}
\def\mykeywords{Travelling Salesman Problem, Symbolic Optimal Control, Coverage, Abstractions, Symbolic Model, Optimal Control, Nonlinear System}
\def\arxiv{}
\documentclass[letterpaper, 10pt, conference]{ieeeconf}
\IEEEoverridecommandlockouts
\usepackage{cite}
\usepackage[english]{babel}
\usepackage{amsmath}
\usepackage{amssymb}
\usepackage{bm}

\usepackage{amsthm}
\usepackage{stmaryrd}
\usepackage{tikz}
\usetikzlibrary{calc,shapes,arrows,automata}
\usepackage[]{caption}
\usepackage{algorithm}
\usepackage{algpseudocode}
\algnewcommand\algorithmicinput{\textbf{Input:}}
\algnewcommand\Input{\item[\algorithmicinput]}
\swapnumbers
\newtheoremstyle{rem}{\topsep}{\topsep}{\normalfont}{0pt}{\bfseries}{.}{ }{\thmname{#1}\thmnumber{#2}\thmnote{ \textup{(#3)}}}
\newtheorem{definition}{Definition}[section]
\newtheorem{theorem}{Theorem}[section]
\theoremstyle{rem}
\newtheorem{example}[definition]{\pushQED{\qed}Example}

\makeatletter
\def\endexample{\popQED\@endtheorem}
\def\begriff#1%
{\@nomath\begriff\ifdim\fontdimen\@ne\font>\z@%
\textbf{#1}\else\textit{#1}\fi}
\def\intcc#1{\ensuremath{\left[#1\right]}}
\def\intco#1{\ensuremath{\left[#1\right[}}
\def\intoc#1{\ensuremath{\left]#1\right]}}
\def\intoo#1{\ensuremath{\left]#1\right[}}

\makeatother
\newcommand{\R}{\mathbb{R}}
\newcommand{\Z}{\mathbb{Z}}
\def\implies{\relax\ifmmode\mathrel{\Rightarrow}\else$\implies$ \fi}
\DeclareMathOperator*{\argmin}{arg\,min}
\usepackage{paralist}
\usepackage{hyperref}
\hypersetup{
colorlinks=true,%
breaklinks=true,%
pdfdisplaydoctitle=true,%
linkcolor={black},%
citecolor={black},
urlcolor={blue},
pdfstartview={FitV},%
pdftitle={\mytitle},
pdfauthor={\myname},%
pdfsubject={Submitted to arxiv on \today.},
pdfkeywords={\mykeywords}
}
\usepackage{subcaption}
\graphicspath{{figures/}} 
\ifx\arxiv\undefined
\else
\usepackage{fancyhdr}
\fi
\title{\bf \LARGE \mytitle}
\author{
\myname
\thanks{
The authors are with the
Munich University of Applied Sciences,
Dept. of Mechanical, Automotive and Aeronautical Eng.,
80335 M\"unchen, Germany.
}%
\thanks{This work has been supported by the German Federal Ministry of Education and Research (Project ARCUS; No. 13FH734IX6). %
}%
}%
\begin{document}
\maketitle

\begin{abstract}
This paper proposes an algorithmic method
to heuristically solve 
the famous Travelling Salesman Problem (TSP)
when the salesman's path evolves 
in continuous state space and discrete time but 
with otherwise arbitrary (nonlinear) dynamics.
The presented method is based on 
the framework of Symbolic Control. 
In this way, our method returns 
a provably correct state-feedback controller 
for the underlying coverage specification, 
which is the TSP leaving out 
the requirement for optimality on the route. 
In addition, we utilize 
the Lin-Kernighan-Helsgaun TSP solver 
to heuristically optimize 
the cost for the overall taken route. 
Two examples, an urban parcel delivery task and 
a UAV reconnaissance mission,
greatly illustrate 
the powerfulness of the proposed heuristic. 
\end{abstract}
\section{Introduction}
\label{s:introduction}
\ifx\arxiv\undefined
\else
\thispagestyle{fancy}
\fi
One of the most prominent problems in 
combinatorial optimization is the 
Travelling Salesman Problem (TSP), 
which \textsc{R. Bellman} formulates as: 
\textit{``A salesman is required to visit 
once and only once each of $N$ different 
cities starting from a base city, 
and returning to this city. 
What path minimizes the total distance travelled
by the salesman?"} \cite{Bellman61}. 
In this paper we aim at a 
fully-automated solution technique
for the TSP when it is posed 
in the \mbox{$n$-dimensional} real space
with the salesman's path following nonlinear dynamics.
To be specific, the salesman dynamics 
are assumed to be given by 
the time-discrete continuous-state control system
\begin{equation}
\label{e:discrete:dynamics}
x(t+1) \in F(x(t),u(t)),
\end{equation}
where 
$x$ is the time-discrete state signal (``salesman") 
taking values in $\mathbb{R}^n$, 
$u$ is the input signal and 
$F$ is a strict set-valued map.
In this dynamic setup 
we interpret the 
$N$ cities of the original problem as 
$N$ target sets in the state space $\mathbb{R}^n$. 
Optimality on the route is understood 
in terms of minimizing 
a prescribed cost functional, 
which is not subject to 
smoothness restrictions and 
may impose hard constraints. 

This generalization of the TSP 
is not an academic playground but
possible applications are divers. 
For example, a UAV reconnaissance mission, 
on the one hand, asks  
for the most efficient sequence 
to visit the areas of interest. 
On the other hand, the dynamical model 
of the vehicle is more complicated than 
reducing its motion to straight lines, 
not to mention obstacles or wind in the mission area. 
Fig.~\ref{fig:uav} illustrates such a mission 
in simulation for which the state-feedback controller 
will result from the contributions of the present work.
Before we outline them
we would like to give a brief literature overview
to the numerous works on the TSP and its variants.
Existing literature can be 
grouped in basically four categories. 
\subsubsection*{Early works -- TSP on networks}
The first category includes works 
which consider the original TSP on networks 
(more precisely, on ordinary 
directed or undirected graphs) and 
investigate algorithms or implementations 
for exactly or approximately solving the problem. 
The earliest works appear around 1955-65, 
e.g. utilizing linear programming 
\cite{DantzigFulkersonJohnson59} 
or dynamic programming 
\cite{Bellman61,
HeldKarp62}.
The Lin-Kernighan heuristic \cite{LinKernighan73} 
is still considered a milestone 
for solving the TSP. A maintained software library for it, 
the \begriff{Lin-Kernighan-Helsgaun solver}, 
is online available \cite{Helsgaun00,Helsgaun17}.
For a comprehensive survey on works 
of this decade see \cite{BellmoreNemhauser68}.
Simple variations of the TSP 
from this time period should also be mentioned, 
e.g. the multiple travelling salesmen problem 
\cite{BellmoreHong74,
Rao80} or the vehicle routing problem \cite{DantzigRamser59}.
\subsubsection*{Advanced variations of the TSP on networks}
Another group of works considers networks 
like the previous works 
but studies solution methods 
for more complex variations of 
the original problem statement.
There is the 
\begriff{Multi-depot Vehicle Routing Problem
with fixed distribution} 
\cite{LimWang05}, 
the \begriff{Heterogeneous Multi-depot Multiple-TSP} 
\cite{SundarRathinam15} or 
the \begriff{Flying Sidekick TSP} 
\cite{MurrayChu15}, 
just to mention a few. 
Others can be found in 
\cite{Bektas06,
ZhangWangYun09,
OberlinRathinamDarbha09,
OberlinRathinamDarbha09b}.
\subsubsection*{The TSP for vehicle dynamics}
A series of other works 
abandons the framework of networks
and poses the TSP for the case of 
vehicle dynamics. 
Therefore, classical solvers for the TSP 
cannot be directly applied to obtain an optimal route.
More concretely, targets are not 
connected by straight lines but 
connecting paths follow nonholonomic planar vehicle dynamics 
like Dubins vehicle \cite{SavlaFrazzoliBullo08,
LeNyFeronFrazzoli12,
AndersonMilutinovic13,
Babel20} 
or the Reed-Shepp vehicle \cite{MalikRathinamDarbhaJeffcoat06}. 
In \cite{GuCaoXieChenSun16} a 3-DOF aircraft model 
is considered and a solution method 
for the multiple-depot-multiple TSP is given, 
where also spatial obstacles are taken into account.
All these works consider indeed nontrivial dynamics, 
yet typically tailor their solutions 
for the particular motion characteristics they assume.
\subsubsection*{Other related works}
Approaching the state of the art from the side of 
motion planning and algorithmic controller synthesis
several works must also be mentioned.
In these contexts the optimality 
on the overall route has not been studied yet
but the underlying coverage specification, i.e. the requirement 
to visit $N$ target sets in any order. 
As an LTL formula, a coverage specification reads 
in its simplest form as\looseness=-1
\begin{equation*}
\diamond \pi_1 \wedge \diamond \pi_2 \wedge \dots \wedge \diamond \pi_N,
\end{equation*}
where $\pi_i$ is the proposition that is true 
if the salesman is in target $i$ \cite{FainekosKressGazitPappas05}.
This class of specifications has been investigated in
\cite{BhatiaKavrakiVardi10,
FainekosKressGazitPappas05,
FainekosGirardKressGazitPappas09}
for second order robot models. 
The special case of two target sets 
for quite general sampled-data control systems 
is investigated in 
\cite{ReissigWeberRungger17,
WeberKreuzerKnoll20,
WeberKnoll20}.

The contribution of this paper is 
a constructive method for 
synthesizing state-feedback controllers
enforcing the said coverage specification 
on plants possessing dynamics \eqref{e:discrete:dynamics}. 
Unlike in existing literature, 
we are allowing for quite general dynamics 
including uncertainties, hard state constraints and 
measurement errors in the closed loop. 
Moreover, the synthesis of the controller 
is performed in a fully automated fashion. 
The presented method is based on 
\begriff{Symbolic Controller Synthesis} 
as established in \cite{ReissigWeberRungger17} and 
the extension of 
\begriff{Symbolic Optimal Control} 
\cite{ReissigRungger18}. 
In addition, we utilize 
the \begriff{Lin-Kernighan-Helsgaun solver} 
\cite{Helsgaun00} 
on top of previous synthesis method
to heuristically determine 
the cheapest sequence for visiting the 
target sets. 
In our experiments we solve two 
Travelling Salesman Problems 
on sampled-data control systems whose 
continuous-time dynamics 
are governed by a differential inclusion.

The rest of this paper is organized as detailed below.
Section \ref{s:notation} contains basic notation. 
Section \ref{s:prelim} provides 
the basic formalism of Symbolic Optimal Control. 
Section \ref{s:tsp} is devoted to 
the rigorous definition of 
the TSP as considered herein.
The main results are presented in Section \ref{s:main}. 
Lastly, Section \ref{s:experiments} includes the simulation results
and Section \ref{s:conclusions} contains conclusions.
\begin{figure}
\centering
\includegraphics[scale=1]{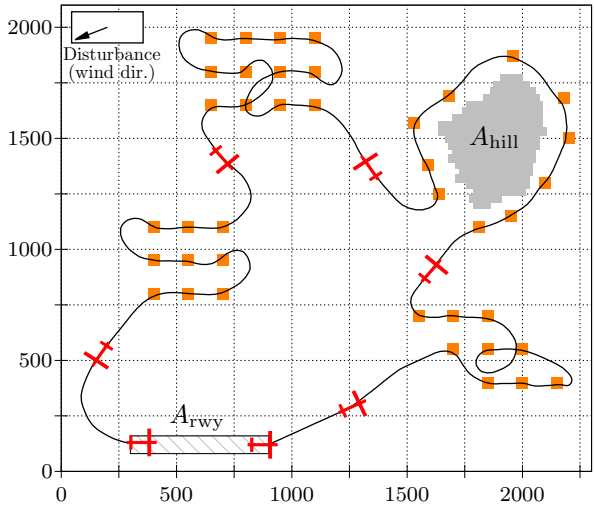}
\caption{\label{fig:uav}Motivating example. 
A UAV (modelled as a Dubins vehicle) shall visit all areas coloured in orange (\textcolor{orange}{$\blacksquare$}) subject to a) starting from and returning to the airfield/runway ($A_\mathrm{rwy}$), b) avoiding obstacles ($A_\mathrm{hill}$), c) minimizing a prescribed cost functional and d) withstanding disturbances. The shown closed-loop trajectory is deduced from the contributions of this work. Details to this example are given in Section \ref{ss:experiments:uav}.}
\vspace{-\baselineskip}
\end{figure}

\section{Notation}
\label{s:notation}
The field of real numbers is denoted by $\mathbb{R}$. 
The subset of non-negative real numbers, of integers and 
of non-negative integers is denoted by 
$\mathbb{R}_+$, $\mathbb{Z}$ and $\mathbb{Z}_+$, respectively. 
E.g. $\mathbb{Z}_+ = \{0,1,2,\ldots\}$. 
For $a, b \in \R$ the open, half-open and closed intervals with endpoints $a, b$
are denoted by $\intoo{a,b}$, $\intoc{a,b}$, $\intco{a,b}$ and $\intcc{a,b}$, respectively, 
and the discrete versions by $\intoo{a;b}$, $\intoc{a;b}$, $\intco{a;b}$ and $\intcc{a;b}$. 
E.g. $\intoc{a;b} = \intoc{a,b} \cap \Z$, $\intoc{1;2} = \{2\}$. 
The symbol $\emptyset$ stands for the empty set. 
The difference of two sets $A$ and $B$ is written as $A \setminus B$. 
The restriction of a function 
$f \colon A \to B$ to $C \subseteq A$ 
is denoted by $f|_C$ \cite{Hungerford74}. 
If $f \colon A \times B \to C$ then $f(\cdot,b)$ 
stands for the map $A \to C$, $a \mapsto f(a,b)$.
The notation $A \rightrightarrows B$ means 
that $f$ is a set-valued map with 
domain $A$ and image the subsets of $B$ \cite{RockafellarWets09}. 
If $f(a) \neq \emptyset$ for all $a \in A$ then $f$ is \begriff{strict}. 
$A^B$ denotes the set of all maps $B\to A$, e.g. an element of $\R^{\intcc{0;4}}$ is a $5$-tuple whose entries are in $\R$. A \begriff{signal} is a map in $A^{\Z_+}$.
\section{Preliminaries}
\label{s:prelim}
The purpose of this section is 
to define the notion of 
control loop and 
optimality on it 
in the version that 
is considered in this work. 
To be specific, the considered class of 
optimal control problems 
is defined in Section \ref{ss:ocp} and 
its solution in Section \ref{ss:ocpsolution}. 
To a large extend, the used concepts are adopted from
\cite{ReissigWeberRungger17,ReissigRungger18,WeberKnoll20}. 
\subsection{System dynamics and optimal control problem}
\label{ss:ocp}
Plants with time-discrete dynamics 
of the form \eqref{e:discrete:dynamics}
are considered, 
where $F \colon X \times U \rightrightarrows X$ is strict and 
$X$ and $U$ are non-empty sets. 
The triple 
\begin{equation}
\label{e:system}
(X,U,F)
\end{equation}
is called \begriff{(transition) system}
with \begriff{state} and 
\begriff{input space} $X$ and $U$, respectively.
Let $S$ denote \eqref{e:system}.
The dynamics \eqref{e:discrete:dynamics} 
induces a \begriff{behaviour initialized at} a state $p \in X$, 
which is the set of all signal pairs $(u,x) \in (U \times X)^{\Z_+}$ such that
$x(0) = p$ and \eqref{e:discrete:dynamics} holds for all $t \in \Z_+$. 
It is denoted by $\mathcal{B}_p(S)$.
In this work, a strict set-valued map
\begin{equation}
\label{e:controller}
\mu \colon \bigcup_{T\in \Z_+}\nolimits X^{\intcc{0;T}} \rightrightarrows U \times \{0,1\}
\end{equation}
is called \begriff{controller},
where the second component of the image, 
called \begriff{stopping signal}, 
indicates if the controller 
is in operation ('0') or 
is disabled ('1') \cite{ReissigRungger13}. 
The closed loop of a system \eqref{e:system} 
interconnected with a controller \eqref{e:controller} 
is formalized by 
the \begriff{closed-loop behaviour initialized at $p \in X$}, 
which is the set of 
all signals 
$(u,v,x) \in (U \times \{0,1\} \times X)^{\mathbb{Z}_+}$ 
satisfying
\begin{equation*}
(u,x) \in \mathcal{B}_p(S) \text{ and } \forall_{t\in \Z_+} : (u(t),v(t)) \in \mu(x_{\intcc{0;t}}).
\end{equation*} 
See Fig.~\ref{fig:controllergeneral} for an illustration. 
Previous set is denoted by $\mathcal{B}_p(\mu \times S)$ and 
the set of all controllers \eqref{e:controller} by $\mathcal{F}(X,U)$.

Operating the controller causes 
costs which depend on 
the time of operation.
Specifically, the cost functional 
\begin{equation}
\label{e:costfunctional:declaration}
J \colon (X \times \{0,1\} \times U)^{\mathbb{Z}_+} \to \mathbb{R}_+ \cup \{\infty\}
\end{equation}
is defined by $J(u,v,x) = \infty$ for infinite operation, 
i.e. if the stopping signal $v$ is identically zero, 
and otherwise by
\begin{equation}
\label{e:costfunctional:definition}
J(u,v,x) = G(x|_{\intcc{0;T}}) + \sum_{t = 0}^{T-1} g(x(t),x(t+1),u(t))
\end{equation}
with ingredients as follows:
\begin{itemize}
\item $T := \inf v^{-1}(1)$, so the first $0$-$1$ edge in $v$ makes the sum finite and defines the termination of the operation.
\item The \begriff{trajectory cost}
\begin{equation}
\label{e:trajectorycost}
G \colon \bigcup_{T \in \Z_+} X^{\intcc{0;T}} \to \mathbb{R}_+ \cup \{ \infty \} 
\end{equation}
rates the full trajectory until stopping.
\item 
The \begriff{running cost}
\begin{equation}
\label{e:runningcost}
g \colon X \times X \times U \to \mathbb{R}_+ \cup \{\infty\}
\end{equation}
occurs in between two consecutive points in time.
\end{itemize}
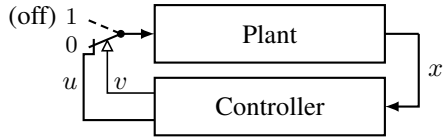
\begin{figure}
\centering
{
\newcommand{\xca}{-3.4}
\newcommand{\yca}{1}
\newcommand{\xcb}{3.4}
\newcommand{\ycb}{2.7}
\newcommand{\xcca}{\xca}
\newcommand{\ycca}{4.2}
\newcommand{\xccb}{\xcb}
\newcommand{\yccb}{5.3}
\newcommand{\xsa}{\xca}
\newcommand{\ysa}{3.15}
\newcommand{\xsb}{\xcb}
\newcommand{\ysb}{4.85}
\newcommand{\ofxa}{2.7}
\newcommand{\ofxb}{1.4}
\newcommand{\ofy}{.5}
\newcommand{\cli}{.4}
\newcommand{\sax}{\xsa-1.4}
\newcommand{\say}{\ysa/2+\ysb/2-.5}
\begin{tikzpicture}[scale=0.45,>=latex]
\draw[thick] (\xsa,\ysa) rectangle node[text width=3cm,align=center]{Plant} (\xsb,\ysb);
\draw[thick] (\xca,\yca) rectangle node[text width=3cm,align=center]{Controller} (\xcb,\ycb);
\draw[->,thick] (\xsb+1,\ysa/2+\ysb/2) |- (\xcb,\yca/2+\ycb/2) ;
\draw[-,thick]  (\xsb+1,\ysa/2+\ysb/2) -- (\xsb,\ysa/2+\ysb/2);
\draw[-,semithick] (\xca,\yca/2+\ycb/2+.4) -| (\sax,\say);
\draw[-,thick]     (\xca,\yca/2+\ycb/2-.4) -| (\xsa-2.1,\ysa/2+\ysb/2-.6) -| (\xsa-1.8,\ysa/2+\ysb/2-.15);
\draw[->,thick] (\xsa-1,\ysa/2+\ysb/2) -- (\xsa,\ysa/2+\ysb/2);
\draw[-,thick] (\xsa-1.95,\ysa/2+\ysb/2-\cli) -- (\xsa-1,\ysa/2+\ysb/2);
\draw[-,thick,densely dashed] (\xsa-1.95,\ysa/2+\ysb/2+\cli) -- (\xsa-1,\ysa/2+\ysb/2);
\draw[semithick] (-.15+\sax,0+\say) -- (.15+\sax,0+\say) -- (0+\sax,.3+\say) -- cycle;
\draw[fill=black] (\xsa-1,\ysa/2+\ysb/2) circle (.1cm);
\node at (\xsb+1.5,\ycb/2+\ysa/2) {$x$};
\node at (\xca-1,\yca/2+\ycb/2+.7) {$v$};
\node at (\xca-2.5,\yca/2+\ycb/2+.7) {$u$};
\node[left] at (\xsa-2.,\ysa/2+\ysb/2-.3) {\small $0$};
\node[left] at (\xsa-2.,\ysa/2+\ysb/2+.5) {(off) \small $1$};
\end{tikzpicture}
}
\caption{\label{fig:controllergeneral}Basic closed loop scheme in Symbolic Optimal Control \cite{ReissigRungger18} with state signal $x$, stopping signal $v$ and input signal $u$.}
\vspace{-\baselineskip}
\end{figure}
The function \eqref{e:costfunctional:declaration} 
is called \begriff{total cost}.
Altogether, the following compact form 
for an optimal control problem 
can be given \cite{ReissigRungger18,WeberKnoll20}.
\begin{definition}
\label{def:ocp}
Consider system \eqref{e:system} and let 
$G$ and $g$ be as in \eqref{e:trajectorycost} and \eqref{e:runningcost}, 
respectively. 
The quintuple
\begin{equation}
\label{e:ocp}
(X, U, F, G, g)
\end{equation}
is called \begriff{optimal control problem}.
\end{definition}
We are interested in finding a controller 
so that the cost for operating 
the closed loop is \emph{finite} and 
ideally minimized for the worst case. 

Returning back to the context of the TSP, 
in Section \ref{s:tsp} we are going to 
define $G$ in \eqref{e:trajectorycost} such that 
it takes the value $\infty$ 
if one target set is missed and 
otherwise $0$. 
Additionally, we wish 
to minimize \eqref{e:costfunctional:declaration} 
along the route by appropriately selecting 
the order of visiting the target sets.

Next, we formalize the aforementioned notion of ``worst case"
and the \begriff{value function} of an optimal control problem.
\subsection{Suboptimal and optimal solutions}
\label{ss:ocpsolution}
Let us relate the optimal control problem \eqref{e:ocp} 
to a \begriff{suboptimal} and \begriff{optimal} solution.
\begin{definition}
\label{def:performance}
Let $\Pi$ be an optimal control problem of the form \eqref{e:ocp}
and let $J$ be the total cost 
\eqref{e:costfunctional:declaration} as defined for $\Pi$.
The map $L$ assigning $(p, \mu) \in X \times \mathcal{F}(X,U)$ to 
\begin{equation}
\label{e:closedloopperformance}
L(p,\mu) := \sup_{(u,v,x) \in \mathcal{B}_p(\mu \times S)} J(u,v,x)
\end{equation}
is called \begriff{performance function} of $\Pi$.  
The map $L(\cdot,\mu)$ is called \begriff{closed-loop performance} of $\mu$.
\end{definition}
Now, the notion of optimality can be defined \cite[Sect.~III.A,VI.C]{ReissigRungger18}.
\begin{definition}
\label{def:valuefunction}
Let $\Pi, L$ be as in Definition \ref{def:performance}.
The \begriff{value function} of $\Pi$ is the map 
$V \colon X \to \mathbb{R}_+ \cup \{\infty\}$ defined by 
\begin{equation*}
V(p) = \inf_{\mu \in \mathcal{F}(X,U)} L(p,\mu).
\end{equation*}
A controller $\mu \in \mathcal{F}(X,U)$ is called \begriff{optimal} 
if $V = L(\cdot, \mu)$. 
An \begriff{optimal solution} of $\Pi$ is such a pair $(V,\mu)$. 
\end{definition}
Suboptimal solutions as defined next 
will also turn out to be satisfactory in applications.
Loosely speaking, a controller is \begriff{suboptimal} if
its closed-loop performance is finite at every state
at which the value function is finite.
\begin{definition}
Let $\Pi$, $L$ and $V$ be as in Definition \ref{def:valuefunction}. Let $A \subseteq X$. 
A controller $\mu \in \mathcal{F}(X,U)$ \begriff{solves $\Pi$ suboptimally on $A$}
if for all $p \in A$ it holds that $V(p) < \infty \implies L(p,\mu) < \infty$.
\end{definition}
\section{\TSP}
\label{s:tsp}
We proceed with the rigorous definition of 
the \TSP{}
as it is considered in this work.
Before that we recall the rigorous formulation of 
the classical TSP on digraphs and its solution \cite{BellmoreNemhauser68}.

\subsection{The classical TSP formulation}
\newcommand{\toursymbol}{\operatorname{Tour}}
Firstly, we define a \begriff{tour}, 
which formalizes the order of 
visiting the $N$ cities 
together with the base city.
In the next definitions, 
the integers $1,\ldots,N$ represent the cities to visit. 
We fix index $1$ for the base city.
\begin{definition}
\label{def:tour}
Let $N \in \mathbb{Z}$, $N \geq 2$.
A finite sequence \[(1,t_2,\ldots,t_N,1)\] where $t_i \in \intcc{2;N}$ and $t_i \neq t_j$ for all $i,j \in \intcc{2;N}$, $i\neq j$ is called \begriff{tour (of length $N$)}.
\end{definition}
For example, for $N=3$ the tour $(1,3,2,1)$ means to visit city $3$, 
then city $2$ and then returning to the base city. 
In general, there are $(N-1)!$ possible tours. Below, the entry $(i,j)$ of the matrix $C$ is the cost to travel from city $i$ to city $j$.\looseness=-1
\begin{definition}
\label{def:classicalTSP}
Let $N$ be as in Definition \ref{def:tour}.
A \begriff{classical \MakeLowercase{\TSP{}}} 
on a weighted directed graph is a tuple $(N,C)$, 
where $C \in \mathbb{R}^{N \times N}$. 
An \begriff{optimal solution of $(N,C)$} is an element of the set
\begin{equation*}
\argmin \Big \{  \sum_{i=1}^N\nolimits C_{t(i),t(i+1)} \ \big | \ t \colon \text{tour of length } N \Big \}.
\end{equation*}
\end{definition}
Next, we transfer previous problem formulation to the continuous 
and dynamic setup that we are considering.
\subsection{\TSP{} formulation in this work}
\label{ss:tsp:definition}
We would like to address some technical details about the 
problem specification to be defined next. 

Firstly, the classical problem formulation
prohibits to visit a city twice except for the base city. 
We will not adopt this requirement 
in favour of a clearer presentation. 
However, that restriction can be easily 
added to our definition.
Secondly, we require the salesman not only 
to visit the targets but also to 
avoid obstacles in the (continuous) state space during travelling. 
This requirement can be encoded in the running cost \eqref{e:runningcost} 
by letting $g$ satisfy $g(x,y,u) = \infty$ whenever $x$ is an element of the obstacle set \cite[Ex.~III.5]{ReissigRungger18}.

In the following problem definition the sets $A_1,\ldots,A_N$ 
take the role of the $N$ cities, where $A_1$ corresponds to the base city (depot). 
Unlike in Definition \ref{def:classicalTSP} 
the cost to travel from $A_i$ to $A_j$ 
is naturally not explicitly given but is determined 
by summing up the running cost $g$ (see \eqref{e:costfunctional:definition}). 
\begin{definition}
\label{def:tsp}
Let $\Pi$ be an optimal control problem 
of the form \eqref{e:ocp} such that $G$
is defined by 
\begin{equation*}
G(x|_{\intcc{0;t}}) = \begin{cases}
0 & \text{if condition \eqref{e:condition} holds }  \\
\infty & \text{otherwise}
\end{cases}
\end{equation*}
where the involved condition is
\begin{equation}
\label{e:condition}
\tag{$\ast$}
\left ( 
\forall_{s \in \{0,t\} } : 
x(s) \in A_1 \right ) 
\wedge 
\left ( 
\forall_{i \in \intoc{1;N}} \exists_{s \in \intcc{0;t}} : 
x(s) \in A_i 
\right )
\end{equation}
with non-empty sets $A_1,\ldots, A_N \subseteq X$.
Then $\Pi$ is called 
\begriff{\TSP{}} with 
\begriff{target sets} $A_1,\ldots, A_N$ 
and \begriff{depot} $A_1$.
\end{definition}
Subsequently, if $S$ denotes the system $(X,U,F)$ then 
\begin{equation*}
\operatorname{TSP}_{S,g}(A_1,\ldots,A_N)
\end{equation*}
stands for the \TSP{} $(X,U,F,G,g)$ with target sets $A_1,\ldots,A_N$ and depot $A_1$.
\section{Controller synthesis algorithm}
\label{s:main}
In this section, our main contributions are presented, 
whose core is given in
Fig.~\ref{fig:hybridcontroller}.
The algorithm is 
presented first and then some remarks 
on implementation and application
to sampled-data control systems are discussed. 
\subsection{Statement and properties of the algorithm}
Before discussing the algorithm, 
we recall quantitative reach-avoid problems 
as they play an important role in the algorithm. 
Roughly speaking,
the key idea is to split the TSP into a sequence of 
special quantitative reach-avoid problems and 
to use an optimality result in \cite{WeberKnoll20}.
\begin{definition}[\!\!\cite{WeberKnoll20}]
\label{def:reachavoid}
Let $\Pi$ be of the form \eqref{e:ocp} such that 
\begin{equation*}
G(x|_{\intcc{0;t}}) = \begin{cases}
G_0(x(t)), & \text{if } x(t) \in A \\
\infty, & \text{otherwise}
\end{cases}
\end{equation*}
defines $G$, where $G_0 \colon X \to \mathbb{R}_+ \cup \{ \infty\}$ and 
$A \subseteq X$ is a non-empty set.
Then $\Pi$ is called \begriff{(quantitative) reach-avoid problem} 
associated with $A$ and $G_0$.
\end{definition}
Below, if $S$ denotes the system $(X,U,F)$ then 
\begin{equation}
\label{e:reachavoid}
\operatorname{Reach}_{S,g}(A,G_0)
\end{equation}
stands for the reach-avoid problem 
$(X,U,F,G,g)$ associated with $A$ and $G_0$.
An optimal controller for \eqref{e:reachavoid} 
can be represented as a strict set-valued map $X \rightrightarrows U \times \{0,1\}$ \cite{ReissigRungger18}.
\makeatletter
\newcommand{\algcolor}[2]{%
  \hskip-\ALG@thistlm\colorbox{#1}{\parbox{\dimexpr\linewidth-2\fboxsep}{\hskip\ALG@thistlm\relax #2}}%
}
\newcommand{\algemph}[1]{\algcolor{gray!20}{#1}}
\makeatother
\begin{figure}
\begin{subfigure}{.48\textwidth}
\hrule{}
\vspace{1ex}
\begin{algorithmic}[1]
\Input{$\operatorname{TSP}_{S,g}(A_1,\ldots,A_N)$}
\State{\label{alg:queue}$Q \gets \{ 1,\ldots, N \}$\hfill{}/\!/~Initialize a queue}
\State{$(A_1',\ldots,A_N') \gets (A_1,\ldots,A_N)$\hfill{}/\!/~Subsets of the targets}
\While{\label{alg:while}$Q\neq \emptyset$}
\State{Pick $i \in Q$}
\State{$Q \gets Q \setminus \{i\}$}
\State{\label{alg:valuefunction}$V_i \gets$ value function of $\operatorname{Reach}_{S,g}(A'_i,0)$} 
\ForAll{$j \in Q$}
\If{$A'_j \setminus V_i^{-1}(\infty) = \emptyset$}
\State{\Return{``Problem can't be solved"}}
\ElsIf{$A'_j \neq A'_j \setminus V_i^{-1}(\infty)$}
\State{$A'_j \gets A'_j \setminus V_i^{-1}(\infty)$ \hfill{}/\!/~Shrink $A'_j$}
\State{$Q \gets Q \cup \{j\}$}
\EndIf{}
\EndFor{}
\EndWhile{\label{alg:while:end}}
\State{\label{alg:tsp}\algemph{$\toursymbol \gets \ $solution of classical TSP $(N,C)$, where \mbox{$C \in \mathbb{R}_+^{N \times N}$} s. that $\forall_{i,j} : C_{i,j} = \min\{V_j( p )\mid p \in A'_i\}$}}
\ForAll{\label{alg:for:controller}$i \in \intcc{1;N}$}
\State{\label{alg:controller}\algemph{$\mu_i \gets$ \footnotesize{}optim. controller of 
$\operatorname{Reach}_{S,g}(A'_{\toursymbol(i)}, V_{\toursymbol(i+1)})$}}
\EndFor{\label{alg:endminusone}}
\State{\Return{$\toursymbol$ and $\mu_1,\ldots,\mu_N$} and $A'_1$}
\end{algorithmic}
\hrule{}
\vspace{1ex}
\caption{\label{fig:hybridcontroller:algorithm}Controller synthesis algorithm. The highlighted lines heuristically optimize the closed-loop performance of the resulting controller in Fig.~\ref{fig:hybridcontroller:formula}.}
\end{subfigure}
\begin{subfigure}{.48\textwidth}
\hrule{}
\vspace{1ex}
\begin{algorithmic}[1]
\Input{$x \in X$}
\Require{\textbf{global integer} $i = 1$}
\If{\label{fig:hybridcontroller:formula:if}$\mu_{\toursymbol(i)}(x)_2 \neq \{0\}$ \textbf{and} $i \leq N$}
\State{$i \gets i + 1$}
\EndIf{\label{fig:hybridcontroller:formula:endif}}
\State{\Return{$\mu_{\toursymbol(i)}(x)$}}
\end{algorithmic}
\hrule{}
\vspace*{1ex}
\caption{\label{fig:hybridcontroller:formula}
Mapping rule of the proposed controller involving the outputs $\operatorname{Tour}$ and $\mu_1,\ldots, \mu_N$ of the algorithm 
in Fig.~\ref{fig:hybridcontroller:algorithm}; The integer $i$ is initialized to $1$ and resides in memory throughout operation.}
\end{subfigure}
\\[.1em]
\begin{subfigure}{.48\textwidth}{
\centering
\newcommand{\xca}{-3.}
\newcommand{\yca}{-.3}
\newcommand{\ycb}{.7}
\newcommand{\xcb}{3.}
\newcommand{\xsa}{\xca}
\newcommand{\ysa}{1.5}
\newcommand{\xsb}{\xcb}
\newcommand{\ysb}{2.9}
\newcommand{\xma}{-8.2}
\newcommand{\xmb}{-7.0}
\newcommand{\ymb}{7}
\newcommand{\ofs}{1.4} 
\newcommand{\ofss}{-5.1} 
\newcommand{\ofsw}{1.4} 
\begin{tikzpicture}[scale=0.55,>=latex]
\draw[thick] (\xma,\ysa+\ofss) rectangle node[rotate=90,text width=3cm,align=center]{switching logic} (\xmb,\ymb+\ofss-.2);
\draw[thick] (\xsa,\ysa) rectangle node[text width=3cm,align=center]{Plant $S$} (\xsb,\ysb);
\draw[thick] (\xca,\yca) rectangle node[text width=3cm,align=center]{Controller $\mu_1$} (\xcb,\ycb);
\draw[thick] (\xca,\yca-\ofs) rectangle node[text width=3cm,align=center]{Controller $\mu_2$} (\xcb,\ycb-\ofs);
\draw[thick] (\xca,\yca-2.7*\ofs) rectangle node[text width=3cm,align=center]{Controller $\mu_N$} (\xcb,\ycb-2.7*\ofs);
\coordinate (v1) at (\xsb,\ysa/2+\ysb/2) ;
\coordinate (v2) at (\xsb+\ofsw,\ysa/2+\ysb/2) ;
\coordinate (v3) at (\xsb+\ofsw,\yca/2+\ycb/2) ;
\coordinate (v4) at (\xsb,\yca/2+\ycb/2) ;
\coordinate (v5) at (\xsb+\ofsw,\yca/2+\ycb/2 - \ofs) ;
\coordinate (v6) at (\xsb,\yca/2+\ycb/2 - \ofs) ;
\coordinate (v8) at (\xsb+\ofsw,\yca/2+\ycb/2 - 2.*\ofs) ;
\coordinate (v9) at (\xsb+\ofsw,\yca/2+\ycb/2 - 2.7*\ofs) ;
\coordinate (v10) at (\xsb,\yca/2+\ycb/2 - 2.7*\ofs) ;
\coordinate (v11) at (\xsb,\ysa/2+\ysb/2+\ofss) ;
\coordinate (v12) at (\xsa,\yca/2+\ycb/2);
\coordinate [draw,circle,scale=0.4] (v13) at (\xsa-\ofsw,\yca/2+\ycb/2);
\coordinate (v14) at (\xsa,\yca/2+\ycb/2-\ofs);
\node [draw,circle,scale=0.4] (v15) at (\xsa-\ofsw,\yca/2+\ycb/2-\ofs) {};
\coordinate (v16) at (\xsa,\yca/2+\ycb/2-2.7*\ofs) ;
\coordinate [draw,circle,scale=0.4] (v17) at (\xsa-\ofsw,\yca/2+\ycb/2-2.7*\ofs);
\node [draw,circle,scale=0.4] (v18) at (-5.75,-2) {};
\coordinate (v20) at (\xsa,\ysa/2+\ysb/2);
\coordinate (v21) at (-5,-1.75) {} {}; 
\coordinate (v22) at (\xma/2+\xmb/2,-3.9) {} ; 
\coordinate (v23) at (-6.,\ysa/2+\ysb/2);
\coordinate (v24) at (\xma/2+\xmb/2,\ymb+\ofss-.2);
\coordinate (v25) at (-5.75,\ysa/2+\ysb/2);
\coordinate (v26) at (\xma/2+\xmb/2,\ysa+\ofss);
\coordinate (v27) at (-6.25,\ysa/2+\ysb/2) {};
\coordinate (v28) at (-6.25,-2.05) {};
\coordinate (v29) at (-6.1,\ysa/2+\ysb/2) {};
\coordinate (v30) at (-6.1,-1.95) {};
\coordinate (v31) at (-5.85,-2.05);
\coordinate (v32) at (-5.85,-1.95);
\coordinate (v33) at (\xsa-\ofsw+.2,\yca/2+\ycb/2);
\coordinate (v34) at (\xsa-\ofsw+.1,\yca/2+\ycb/2-2.7*\ofs+.1);
\coordinate (v35) at (\xsa-\ofsw,\yca/2+\ycb/2-\ofs-.15);
\draw[<-,thick] (v24) |- (v27) -- (v28) -- (v31);
\draw[->] (v32) -- (v30) |- (v20);
\draw[->] (v1) -- (v2) -- (v3) -- (v4);
\draw[->] (v3) -- (v5) -- (v6);
\draw[->] (v8) -- (v9) -- (v10);
\draw[-] (v12) -- (v13);
\draw[-] (v14) -- (v15);
\draw[-] (v16) -- (v17);
\draw[thick,-] (v18) -- (v35);
\draw[dashed,thick,-] (v18) -- (v33);
\draw[dashed,thick,-] (v18) -- (v34);
\draw[thick,->] (v26) -- (v22) -| (v21); 
\node at (\xsa/2+\xsb/2,\yca/2+\ycb/2-1.75*\ofs) {$\vdots$};
\node at (-3.75,\yca/2+\ycb/2-1.75*\ofs) {$\vdots$};
\node at (\xsb+\ofsw,-1.75) {$\vdots$};
\node at (3.5,\ysb-.4) {$x$};
\node at (-4.5,\ysb-.4) {$u$};
\node at (-7,\ysb-.4) {$v$};
\node at (-3.75,0.6) {\footnotesize$(u,v)$};
\node at (-3.75,0.6-\ofs) {\footnotesize$(u,v)$};
\node at (-3.75,0.6-5*\ofs/2-.35) {\footnotesize$(u,v)$};
\end{tikzpicture}
\\[.1em]
}
\caption{\label{fig:hybridcontroller:illustration}Illustration of the closed loop. The block ``switching logic" corresponds to lines \ref{fig:hybridcontroller:formula:if}--\ref{fig:hybridcontroller:formula:endif} in Fig.~\ref{fig:hybridcontroller:formula}.}
\end{subfigure}
\caption{\label{fig:hybridcontroller}Proposed controller synthesis algorithm and resulting controller for solving the Travelling Salesman Problem.}
\end{figure}
Before rigorously formulating the properties 
of the algorithm in Fig.~\ref{fig:hybridcontroller:algorithm},
a rough description is given.

\textbf{\textit{First part (lines \ref{alg:queue}--\ref{alg:while:end}).}} 
This part of the algorithm is a fixed-point iteration. 
In case of success, i.e. in case line \ref{alg:tsp} is reached, 
non-empty subsets $A_i'$ of the target sets $A_i$ are found
for each $i \in \intcc{1;N}$ such that the following holds:
For every $i,j \in \intcc{1;N}$, $i \neq j$ 
an optimal controller for $\operatorname{Reach}_{S,g}(A_j',0)$ successfully steers 
any state in $A_i'$ to $A_j'$. 
In other words, the underlying coverage specification is solved.

\textbf{\textit{Second part (line \ref{alg:tsp}).}} 
This part heuristically optimizes the order of 
visiting the target sets as follows. 
The cost for reaching $A_j'$ starting from $A_i'$ 
is optimistically estimated using 
the previously calculated value function $V_j$ and 
stored in the entry $(i,j)$ of the matrix $C$. 
The resulting classical TSP is then solved. 
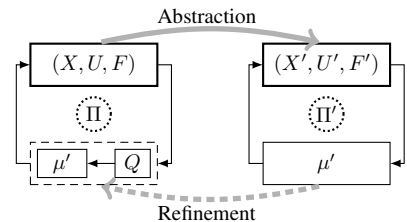
\begin{figure}
\centering
\usetikzlibrary{arrows}
\begin{tikzpicture}[scale=.75, every node/.style={scale=.8}]
\draw[thick]  (-3.75,3.25) rectangle (-1.5,2.5) node[pos=.5] {$(X,U,F)$};
\draw[thick]   (0.375,3.25) node (v2) {} rectangle (2.625,2.5) node[pos=.5] {$(X',U',F')$};
\draw[-latex] (-1.5,2.875) -- (-1.25,2.875) -- (-1.25,1.125) -- (-1.5,1.125);
\draw  (-2.25,1.375) rectangle (-1.625,0.875) node [pos=.5] (v4) {$Q$};
\draw  (-3.625,1.375) rectangle (-2.75,0.875) node[pos=.5] {$\mu'$};
\draw[thick,densely dotted] (-2.625,2) circle (0.3) node {$\Pi$};
\draw[thick,densely dotted] (1.5,2) circle (0.3) node {$\Pi'$};
\draw[-latex] (-2.25,1.125) -- (-2.75,1.125);
\draw[-latex] (-3.75,1.125) -- (-4,1.125) -- (-4,2.875) -- (-3.75,2.875);
\node (v1) at (-2.625,3.25) {};
\node (v4) at (1.5,3.25) {};
\draw[gray!60!white,line width=2][->]  (v1) edge [ bend angle=12,bend left] (v4);
\node at (-0.625,3.75) {Abstraction};
\draw  (0.375,1.5) rectangle (2.625,0.75) node[pos=.5] {$\mu'$};
\draw[-latex] (2.625,2.875) -- (2.875,2.875) -- (2.875,1.125) -- (2.625,1.125);
\draw[-latex] (0.375,1.125) -- (0.125,1.125) -- (0.125,2.875) -- (0.375,2.875);
\node (v3) at (0.375,0.75) {};
\node (v5) at (1.5,0.75) {};
\node (v6) at (-2.625,0.75) {};
\draw[gray!60!white,line width=2,densely dashed][->]  (v5) edge [ bend angle=12,bend left] (v6);
\node at (-0.625,0.25) {Refinement};
\node (v7) at (-1.5,0.75) {};
\draw[densely dashed]  (-3.75,1.5) rectangle (v7);
\end{tikzpicture}
\caption{\label{fig:symboliccontrol}
Principle of symbolic controller synthesis \cite{ReissigWeberRungger17,WeberKnoll20}.}
\vspace{-3ex}
\end{figure} 

\textbf{\textit{Third part (lines \ref{alg:for:controller}--\ref{alg:endminusone}).}} 
In line \ref{alg:controller}
controllers that are optimal for visiting two target sets in succession
are calculated by utilizing \cite[Th.~III.1]{WeberKnoll20}. 
More concretely, assuming, e.g., $\toursymbol(N) = N$ then $\mu_N$ is \emph{optimal} for visiting \underline{first} $A_N$ \underline{and then} $A_1$. This is why the matrix $C$ contains \emph{optimistic} estimates. 
In general, $\mu_N$ does not coincide 
with an optimal controller for 
$\operatorname{Reach}_{S,g}(A_N,0)$ \cite[Sect.~IV]{WeberKnoll20}.
This step significantly reduces the total cost in practice. 
See Section \ref{ss:experiments:uav}.\looseness=-1

The formal properties of the algorithm are stated next.
\begin{theorem}
Let $S$ be a system of the form \eqref{e:system}, 
let $g$ be as in \eqref{e:runningcost} and 
let $A_1,\ldots,A_N \subseteq X$ be non-empty sets, $N>1$. 
Let $A_1'$ and 
$\mu_1, \ldots, \mu_N$ 
be obtained by 
the algorithm shown in Fig.~\ref{fig:hybridcontroller:algorithm}
when applied to $\Pi:=\operatorname{TSP}_{S,g}(A_1,\ldots,A_N)$.
The controller $\bm{\mu}\in \mathcal{F}(X,U)$
defined by Fig.~\ref{fig:hybridcontroller:formula} 
solves $\Pi$ suboptimally on $A_1'$.
\end{theorem}
\begin{proof}
When the algorithm 
arrives at line \ref{alg:tsp} 
the value functions $V_i$ are finite on $A'_j$ 
for all $i,j \in \intcc{1;N}$. 
Then note that changing line \ref{alg:tsp} 
to the trivial tour
\begin{equation*}
\toursymbol \gets (1,2,\ldots,N,1)
\end{equation*}
and line \ref{alg:controller} to
\begin{equation}
\label{e:lineofalgorithm}
\mu_i \gets \text{optimal controller of } \operatorname{Reach}_{S,g}(A'_{\toursymbol{}(i)},0)
\end{equation}
does not change 
the statement of the theorem. 
Denote by $\Pi_i$ the optimal control problem 
involved in \eqref{e:lineofalgorithm}. 
Note that $V_i = L_i(\cdot,\mu_i)$ for all $i \in \intcc{1;N}$, 
where $L_i$ is the performance function of $\Pi_i$.
Let $p \in A'_1$ and 
$(u,v,x) \in \mathcal{B}_p(\bm{\mu} \times S)$. 
Then $1 \in \bm{\mu}(x|_{\intcc{0;t}})_2 = \mu_2(x|_{\intcc{0;t}})_2$ for some $t\in \Z_+$
since $\mu_2$ solves $\Pi_2$ optimally. 
So $x(t) \in A'_2$. 
By induction,
for all $i \in \intcc{2;N}$ 
there exists $s \in \Z_+$ such that $x(s) \in A'_i$. 
Since $\mu_1$ solves $\Pi_1$ optimally the proof is completed.
\end{proof}
The controller $\bm{\mu}$ and the resulting closed loop are 
illustrated in Fig.~\ref{fig:hybridcontroller:illustration}.
\subsection{Implementation and application of the algorithm}
We would like to comment on 
implementing the algorithm 
in Fig.~\ref{fig:hybridcontroller:algorithm}.
First, for solving the classical TSP in line \ref{alg:tsp} 
we propose the use of 
the \begriff{Lin-Kernighan-Helsgaun solver} \cite{Helsgaun00}.
Clearly, any other solver can also be used.
Second, lines \ref{alg:valuefunction} and \ref{alg:controller}
require to solve quantitative reach-avoid problems on the system $S$.
For the case that the state and input spaces of $S$ are \emph{finite}
such algorithms exist \cite{ReissigRungger18,MacoveiciucReissig19,WeberKreuzerKnoll20}.

Consequently, in combination with the 
principle of \begriff{symbolic controller synthesis} 
\cite{ReissigWeberRungger17,ReissigRungger18} 
our synthesis algorithm can be applied 
to sampled-data control systems, whose 
dynamics are of the form
\begin{equation}
\label{e:system:cont}
\dot \xi(t) \in f(\xi(t),u(t)) + W.
\end{equation} 
In \eqref{e:system:cont}, 
$\xi$ is the state signal, 
$u$ is the input signal taking values in 
$U \subseteq \mathbb{R}^m$, 
$f\colon \mathbb{R}^n \times U \to \mathbb{R}^n$ and 
$W \subseteq \mathbb{R}^n$ is a set accounting for
disturbances.
By means of sampling a formulation as 
a system \eqref{e:system} with $X=\mathbb{R}^n$ 
is possible under certain assumptions on $f$ and $W$ 
\cite[Sect.~VIII.A]{ReissigWeberRungger17}.

The approach to synthesize controllers for 
an optimal control problem associated with 
a sampled system is as follows \cite{ReissigRungger18}.
The original optimal control problem $\Pi = (X,U,F,G,g)$ 
is transferred to 
an \begriff{abstract} optimal control problem 
$\Pi' = (X',U',F',G',g')$.
The involved \begriff{discrete abstraction} $(X',U',F')$ of 
the sampled system associated with \eqref{e:system:cont}
has \emph{finite} state and input space, $X'$ and $U'$. 
(Typically, $X'$ is a cover of $X$ such that most of its elements
are translations of $\intcc{0,\eta_1} \times \ldots \times \intcc{0,\eta_n}$. 
The vector $\eta \in \mathbb{R}_+^n$ is called \begriff{grid parameter}, 
which will be mentioned later in the experimental results.)
In the case that $\Pi'$ can be solved 
the obtained controller $\mu'$ for $\Pi'$ is
refined to a controller for $\Pi$. 
This refinement step requires for 
the methodology of \cite{ReissigWeberRungger17} only 
the interconnection with a simple quantizer. 
See Fig.~\ref{fig:symboliccontrol}.\looseness=-1
\section{Experimental results}
\label{s:experiments}
Subsequently, two examples are presented, 
which greatly demonstrate 
the powerfulness of the proposed heuristic. 
Both examples are related to ``vehicles"
since \TSPs{}
naturally are most intuitive 
when the target sets 
have a spatial component.
Nevertheless it is worth pointing out  
once more that the presented method 
can be applied to 
any transition system with dynamics \eqref{e:discrete:dynamics}.
\subsection{Reconnaissance mission}
\label{ss:experiments:uav}
Firstly, the reconnaissance mission 
with an uninhabited aerial vehicle (UAV)
that was mentioned as a motivation in
\mbox{Section \ref{s:introduction}} is investigated.
The mission is illustrated in Fig.~\ref{fig:uav}.
\subsubsection{Control problem}
The dynamics of Dubins vehicle \cite{LaValle06}
are assumed for the UAV 
with additional disturbances, 
i.e. dynamics \eqref{e:system:cont} with 
$W = \intcc{-5,5} \times \intcc{-2,2} \times \intcc{-0.04,0.04}$ and 
$f \colon \mathbb{R}^3 \times U \to \R^3$ defined by 
$U = \intcc{20,50}\times \intcc{-0.5,0.5}$,
\begin{equation*}
f(x,u) = \big (u_1 \cos(x_3), u_1 \sin(x_3), u_2 \big ).
\end{equation*}
Thus, the planar position of the UAV is 
described by $(x_1,x_2)$ and 
$x_3$ is its heading.
The control inputs $u_1$ and $u_2$ are 
the velocity and the angular velocity, 
respectively.
By theory, a time-discrete version
of the continuous dynamics needs to be considered,
which is the transition system $S = (\mathbb{R}^3, U , F)$
defined as the sampled system associated with \eqref{e:system:cont}
and sampling period $\tau = 0.65$ \cite[Def.~VIII.1]{ReissigWeberRungger17}.

The TSP of the form \eqref{e:ocp} with targets 
$A_1,A_2,\ldots, A_{41}$ and depot $A_1$ is to be solved on $S$, where $A_1 := A_{\mathrm{rwy}}$,
\begin{align*}
A_{\mathrm{rwy}} &= 
\intcc{300,900} \times 
\intcc{80,160} \times 
\intcc{-10^\circ,10^\circ}, \quad \text{(``runway")}\\
A_2 &= \intcc{375,425} \times \intcc{775,825} \times \R, \quad \text{(``area of interest")} \\
A_3 &= \intcc{525,575} \times \intcc{775,825} \times \R. \quad \text{(``area of interest")}
\end{align*}
The other target sets are translations of $A_2$, 
which are positioned as depicted in Fig.~\ref{fig:uav}. 
The running cost of the mission compromises between minimum time and 
small absolute angular velocities. 
Moreover, it includes the requirement 
of avoiding obstacles (cf. Section \ref{ss:tsp:definition}).
Specifically,
\begin{equation*}
g(x,y,u) = \begin{cases}
\infty, & \text{if } x \in (\R^3 \setminus X_{\mathrm{mis}}) \cup A_{\mathrm{nofly}} \cup A_{\mathrm{hill}} \\
\tau + u_2^2, & \text{otherwise (} u_2 \text{ in radians)}
\end{cases}
\end{equation*}
defines $g$ in \eqref{e:ocp}, 
where $A_{\mathrm{hill}} \subseteq \R^3$ is 
a spatial obstacle set as indicated in Fig.~\ref{fig:uav}, 
$X_{\mathrm{mis}} = \intcc{0,2500} \times \intcc{0,2200} \times \R$ is the mission area and
$A_{\mathrm{nofly}} = \intcc{320,880} \times \intcc{100,140} \times \intcc{12^\circ, 348^\circ}$
forces a proper approach to the airfield.
\subsubsection{Heuristic solution}
To solve the defined TSP, 
the algorithm in Fig.~\ref{fig:hybridcontroller} 
is applied 
and the controller 
shown in Fig.~\ref{fig:hybridcontroller:formula} 
controls the UAV.
To apply the algorithm, 
a discrete abstraction $(X',U',F')$ for $S$ is computed. 
Computational details are given in Tab.~\ref{tab:uav}.

The cost for the closed-loop trajectory 
shown in Fig.~\ref{fig:uav}, 
which is subject to the disturbance $w := (-5,-2,0) \in W$, 
is $348.5$. 
Without the use of the controller improvement in 
lines \ref{alg:for:controller}--\ref{alg:endminusone}, 
which would reduce runtime by $53\%$, 
the cost would be $440.3$ or $26\%$ more. 
Fig.~\ref{fig:uav:detail} indicates
what goes wrong in this case, 
namely targets sets are reached without taking into account that 
another target set will follow. See \cite{WeberKnoll20}.\looseness=-1

The synthesized controller is by design 
robust against any disturbance vector in $W$. 
Fig.~\ref{fig:uav:detail} depicts
a closed-loop trajectory when the disturbance $-w$ is acting on the UAV.
\begin{table}
\centering
\begin{tabular}{|l|l|}
\hline
Quantity & Value (description) \\
\hline
\hline
$|X'|$ & $41.25 \cdot 10^6$ (grid parameter $(0.1,0.1,2\pi/75)$) \\
$|U'|$ & $49$ ($7 \cdot 7$ values of $U$) \\
\hline
Runtime lines \ref{alg:queue}-\ref{alg:while:end} & $69$ min. (using \cite{WeberKreuzerKnoll20}) \\
Runtime line \ref{alg:tsp} & $< 1$ sec. (using \texttt{LKH-2.0.9} \cite{Helsgaun00}) \\
Runtime lines \ref{alg:for:controller}-\ref{alg:endminusone} & $78$ min. (using \cite{WeberKreuzerKnoll20}) \\
Total runtime & $147$ min. \\
Total RAM usage & $6.7$ GB \\
\hline
\end{tabular}
\caption{\label{tab:uav}Computational details to Section \ref{ss:experiments:uav}. 
The implementation is written in C. 
Computations are executed in parallel with $48$ cores on x86-64 SuSE Linux (Intel Xeon E5-2697 v3, 2.6 GHz).}
\end{table}
\begin{figure}
\hfill~\includegraphics[scale=1]{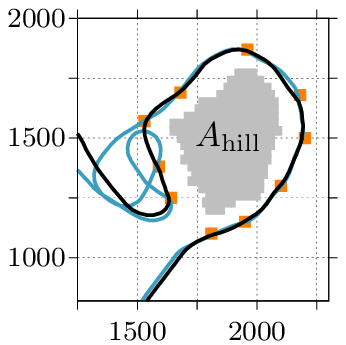}\hfill
\includegraphics[scale=1]{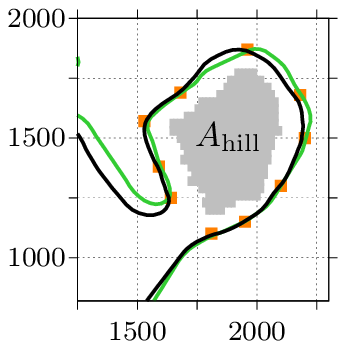}\hfill~
\caption{\label{fig:uav:detail}Detail to the reconnaissance mission in Section \ref{ss:experiments:uav}. The black-coloured trajectory is the same as in Fig.\ref{fig:uav}. The blue-coloured trajectory is obtained when line \ref{alg:controller} is replaced by \eqref{e:lineofalgorithm}. The green-coloured trajectory is subject to the opposite disturbance acting on the black-coloured trajectory.}
\vspace*{-.4cm}
\end{figure}
\subsection{Urban parcel delivery}
\label{ss:parceldelivery}
\begin{figure*}
\hfill{}
\centering
\begin{minipage}{0.6\textwidth}
\centering
\includegraphics[scale=1]{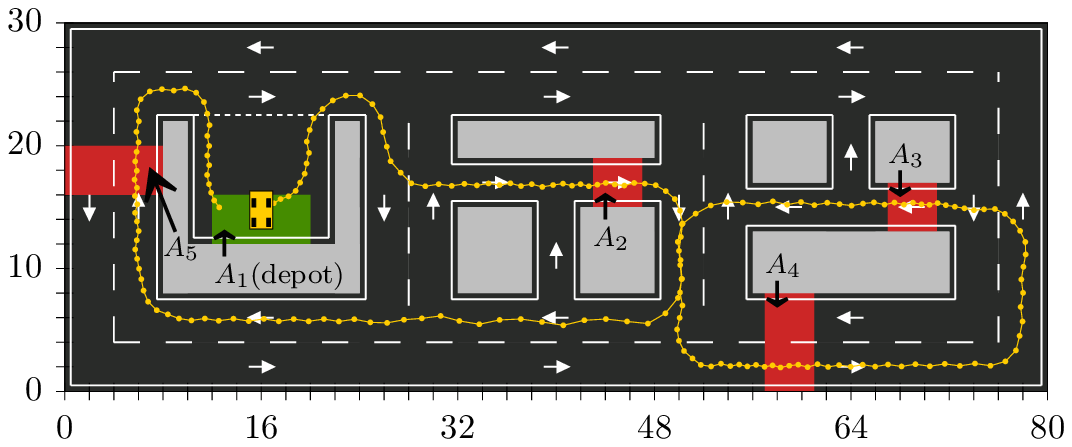}
\caption{\label{fig:vehicle}
Illustration of the closed-loop trajectory 
obtained through the proposed algorithm 
(Fig.~\ref{fig:hybridcontroller:algorithm}). 
The initial state $(16,14,\pi/2,5)$ of the trajectory is indicated by the car symbol.}
\end{minipage}
\hfill{}
\begin{minipage}{0.33\textwidth}
\centering
\captionsetup{type=table} %
\begin{tabular}{cr} 
  Tour & Total cost \\ \hline \hline
  $\mathbf{(1,2,4,3,5,1)}$ & $\mathbf{94.2}$ \\
  $(1,2,3,4,5,1)$ & $104.1$ \\
  $(1,5,2,4,3,1)$ & $106.3$ \\
  $(1,4,3,2,5,1)$ & $107.6$ \\
  $(1,5,4,3,2,1)$ & $109.1$ \\
  $(1,3,4,2,5,1)$ & $112.1$ \\
  $(1,2,5,4,3,1)$ & $114.8$ \\
  others & $> 114.8$ \\ \hline
    \phantom{1} & \phantom{1}
  \end{tabular}
\caption{\label{tab:vehicle:costs}Comparison of the total cost \eqref{e:costfunctional:definition} of 
the closed-loop trajectory initialized at 
$(16,14,\pi/2,5)$ for all possible tours. 
The tour in the first row corresponds 
to the trajectory shown in Fig.~\ref{fig:vehicle}.}
\end{minipage}
\hfill{}
\vspace{-\baselineskip}
\end{figure*}
The second scenario is a 
delivery task in an urban environment 
as depicted in Fig.~\ref{fig:vehicle}:
A delivery truck has to visit 
the four areas $A_2, \ldots, A_5$ coloured in red 
starting from and 
returning to the depot $A_1$ coloured in green. 
Details are given below. 
Parts of the scenario are taken from \cite[Sect.~IV.A]{WeberKnoll20}.

\subsubsection{Control problem}
The equations of motion of the truck 
are given by \eqref{e:system:cont} and
include four states, 
which are the planar position $(x_1,x_2)$, 
the orientation $x_3$ and the velocity $x_4$. 
The control inputs are the
acceleration $u_1$ and 
steering angle $u_2$ of the truck. 
The map $f$ and the set $W$ in \eqref{e:system:cont} 
are given by
\begin{equation}
\label{e:examples:vehicle:dynamics}
f\colon \R^4 \times U \to \R^4, 
f(x,u) = \begin{pmatrix}
x_4 \cdot \cos(\alpha + x_3) \cdot \beta \\
x_4 \cdot \sin(\alpha + x_3) \cdot \beta \\
x_4 \cdot \tan(u_2) \\
u_1
\end{pmatrix},
\end{equation}
where $U = \intcc{-6,4} \times \intcc{-0.5,0.5}$, $\alpha = \arctan(\tan(u_2)/2)$, 
$\beta =  \cos(\alpha)^{-1}$, $W=\{(0,0)\} \times \intcc{-0.01,0.01} \times \intcc{-0.1,0.1}$.
We consider the sampled system \eqref{e:system} with $X=\R^4$
associated with \eqref{e:system:cont} and sampling period $\tau = 0.1$.
The delivery task is defined as 
the TSP \eqref{e:ocp} with targets sets 
$A_1,\ldots,A_5$ and 
depot $A_1$, where
\begin{align*}
A_1 &= \intcc{12,20} \times \intcc{12,16} \times (\intcc{0,2\pi} \setminus I_{\mathrm{south}})\times \intcc{0,7}, \\ 
I_{\mathrm{south}} &= 3\pi/2 + \intoo{-3\pi/8,3\pi/8}, \\
A_2 &= \intcc{43,47} \times \intcc{15,19} \times \intcc{0,2\pi} \times \intcc{0,7}.
\end{align*}
Both $A_1$ and $A_2$ limit speed 
while $A_1$ additionally 
prohibits a truck orientation to the south.
The other targets sets are similar to $A_2$ and 
can be identified from Fig.~\ref{fig:vehicle}.

The running cost $g$ satisfies $g(x,y,u) = \infty$ in two cases: 
Firstly, if $x$ is in the obstacles set $(\R^4 \setminus \bar X) \cup O$, where 
$O$ is the union of the grey-coloured sets in Fig.~\ref{fig:vehicle} and 
\begin{equation*}
\bar X = \intcc{0,80} \times \intcc{0,30} \times \intcc{0,2\pi} \times \intcc{0,18}.
\end{equation*}
Secondly, if $x$ violates the common right-hand traffic rules: 
the traffic rules are not violated if, e.g., $x$ is in
\begin{align*}
\intcc{0,80} \times \intcc{26,30} \times I_{\mathrm{west}} \times \intcc{0,18}, I_{\mathrm{west}} = \pi + \intcc{-\tfrac{3\pi}{8},\tfrac{3\pi}{8}}
\end{align*}
which is the northernmost lane, or if $x$ is in
\begin{equation*}
\intcc{10,22} \times \intcc{12,30} \times \intcc{0,2\pi} \times \intcc{0,18},
\end{equation*}
which are states in proximity of the depot.
In the finite case, 
$g$ balances minimum time and proper driving style, i.e.
\begin{equation*}
g(x,y,u) = \tau + u_2^2 + \min_{m\in M}\| (y_1,y_2) - m \|_2.
\end{equation*}
Here, $M \subseteq \R^2$ describes the axes of the roadways, e.g. 
$\intcc{2,78} \times \{28\} \subseteq M$, and the traffic guidance into the depot by\looseness=-1
\begin{equation*}
\big ( \{12\} \times \intcc{14,28} \big ) \cup \big (\{20\} \times \intcc{14,28} \big ) \cup \big ( \intcc{12,20} \times \{14\}\big ) \subseteq M.
\end{equation*}
\subsubsection{Heuristic solution}
A discrete abstraction $(X',U',F')$ is computed, 
where $X'$ possesses the grid parameter
$(8/15,30/57,2\pi/62,9/25)$ and 
$U'$ consists of $8\cdot 10$ values of $U$. 
The total runtime to solve the problem 
with the algorithm in Fig.~\ref{fig:hybridcontroller:algorithm} 
is $3$ hours using $21$ GB RAM.
The total cost for the closed-loop trajectory shown in Fig.~\ref{fig:vehicle} is $94.2$. 
Tab.~\ref{tab:vehicle:costs} lists the costs for other possible tours and 
confirms that the heuristic we proposed returns the cheapest tour.
\section{Conclusions}
\label{s:conclusions}
We considered a generalization of the \TSP{}, where the salesman's path
evolves subject to continuous-state discrete-time dynamics 
with possible uncertainties.
By subdividing the problem into several special 
(quantitative) reach-avoid problems 
we succeeded in synthesizing controllers
steering the salesman heuristically optimal to its targets.
Formally, the obtained controllers are correct-by-design 
ensuring that the involved coverage specification is enforced, at least qualitatively, on the closed loop.

Finally, we would like to point out that our method can be easily 
extended to a similar generalization of the \begriff{Multiple \TSP{}} \cite{BellmoreHong74}. 
This requires only to slightly generalize line \ref{alg:tsp}
of the algorithm in Fig.~\ref{fig:hybridcontroller:algorithm}. 
Then, for example, also 
control policies for multi-UAV missions can be synthesized.
\section*{Acknowledgements}
The authors would like to thank Dominik K\"unzel for the discussion on existing literature on the TSP and the Leibniz Supercomputing Centre for providing the compute resources.

\bibliography{../../BibTex/mrabbrev,../../BibTex/articles,../../BibTex/books,../../BibTex/online}
\end{document}